\documentclass[a4paper,12pt,reqno,oneside]{amsart} 

\usepackage{setspace} 
\usepackage{typearea} 
\usepackage{amscd,amssymb,verbatim} 
\usepackage
{graphicx} \setkeys{Gin}{width=\linewidth}
\usepackage{enumitem}

\usepackage{xr-hyper}
\usepackage{cite}
\usepackage[colorlinks,final,bookmarksnumbered,bookmarks]{hyperref}
\usepackage{amsrefs}

\usepackage{ifthen}
\newcommand{\arxiv}[2][]{\ifthenelse{\equal{#1}{}}
{\href{http://arxiv.org/abs/#2}{\tt arXiv:#2}}
{\href{http://arxiv.org/abs/math/#2}{\tt arXiv:math.#1/#2}}}

\usepackage[utf8]{inputenc}
\usepackage[T1,T2A]{fontenc}
\usepackage[russian]{babel}

\theoremstyle{plain}
\newtheorem{theorem}{Теорема}[section]
\newtheorem{lemma}[theorem]{Лемма}
\newtheorem{corollary}[theorem]{Следствие}
\newtheorem{proposition}[theorem]{Предложение}

\newtheoremstyle{remark}
{}{}{}{}{\itshape}{}{ }{\thmname{#1}\thmnumber{ \itshape #2.}}
\theoremstyle{remark}
\newtheorem{remark}[theorem]{Замечание}
\newtheorem{example}[theorem]{Пример}

\newtheoremstyle{concise}
{}{}{}{}{\bfseries}{}{ }{\thmnumber{#2.}\thmnote{ #3.}}
\theoremstyle{concise}

\def\x{\times}
\def\but{\setminus}

\def\eps{\varepsilon}
\def\phi{\varphi}

\renewcommand{\:}{\colon}

\def\R{\Bbb{R}}
\def\Z{\Bbb{Z}}

\def\C{\Bbb{C}}

\DeclareMathOperator{\id}{id}

\makeatletter
\def\@settitle{\begin{center}%
    \baselineskip14\p@\relax
    \bfseries
    \@title
  \end{center}%
}
\makeatother

\begin{document}
\title[]{ПОГРУЖЕНИЯ ОКРУЖНОСТИ В ПОВЕРХНОСТЬ}
\author{С. А. Мелихов}
\address{Математический институт им. В. А. Стеклова РАН, 119991 Москва, ул. Губкина 8}
\email{melikhov@mi-ras.ru}
\thanks{Исследование выполнено за счет гранта Российского научного фонда (проект № 14-50-00005)}

\begin{abstract}
Погружения $f$ окружности в двумерное многообразие $M$ расклассифицированы в терминах элементарных инвариантов: 
чётности $S(f)$ числа двойных точек самотрансверсальной $C^1$-аппроксимации $f$
и числа вращения $T(e\bar f)$ погружения $e\bar f\:S^1\to M_f\subset\R^2$, где $\bar f$ --- поднятие $f$ 
в накрытие $M_f$ поверхности $M$, соответствующее подгруппе $\left<[f]\right>\subset\pi_1(M)$.

А именно, погружения $f,g\:S^1\to M$ регулярно гомотопны, если и только если они гомотопны, и если $M=S^2$ или
$\R P^2$ или нормальное раслоение $\nu(f)$ неориентируемо, то $S(f)=S(g)$; а если $M\ne S^2,\R P^2$ и
$\nu(f)$, $\nu(g)$ имеют ориентации $o$, $o'$, согласованные относительно гомотопии, то 
$T(e_o\bar f)=T(e_{o'}\bar g)$, где $e_o$ --- стандартное вложение ориентированной поверхности $M_f$ 
(кольца или плоскости) в $\R^2$.

В действительности, для гомотопных погружений $f$, $g$ как $S(f)-S(g)$, так и $T(e_o\bar f)-T(e_{o'}\bar g)$ 
сводятся к числу вращения поднятия некоторого нульгомотопного погружения $f\#g^*$ на универсальное накрытие $M$.

Выше <<погружения>> $S^1\to M$ --- гладкие или топологические; приводится теорема сглаживания, показывающая, что 
это не имеет значения.
Также получена классификация погружений графа в $M$ с точностью до регулярной гомотопии в терминах инвариантов
$S(f)$ и $T(e_o\bar f)$ погруженных окружностей.
Доказательства основаны на h-принципе и несложны.
\end{abstract}

\maketitle

\section{Введение}

Гомотопическая классификация гладких погружений окружности в поверхности была получена С. Смейлом в 1958г.
и с тех пор многократно обобщалась и передоказывалась (см.\ \cite{Ad}, \cite{RS}, \cite{RS2}).
Этим классическим результатом, однако, ещё далеко не решается вопрос о классификации таких погружений в 
элементарных терминах (например, наглядных или легко вычислимых на практике).

В недавней работе Д. Пермякова \cite{P}, развивающей методы Reinhart'а (1960), Chillingworth'а (1972) и 
Бурмана--Поляка \cite{BP2}, была получена классификация гладких (в смысле \cite{Mu}) погружений графа в двумерное 
многообразие $M$, отличное от проективной плоскости, в терминах некоторого <<числа вращения>> пары гомотопных погружений 
$S^1\to M$ по отношению к заданному векторному полю с нулями на $M$.
Этот же результат в немного усиленном виде (в части понимания того, что именно считать <<классификацией>>) вошёл
в диссертацию Пермякова \cite{P2}, которую мне выпало оппонировать.
Не имея никаких претензий к диссертации (автор которой, несомненно, решил поставленную ему задачу), я в то же время
не мог не заметить, что с точки зрения принципа <<лезвия Оккама>> сам результат, опубликованный в Матсборнике \cite{P},
не выглядит вполне убедительным.

Число вращения погружения $f\:S^1\to\R^2$ --- элементарный классический инвариант, определяемый в одну строчку 
(а именно, это --- число полных оборотов касательного вектора) и не зависящий ни от каких выборов и векторных полей.
Однако, предложенное в \cite{P} и \cite{P2} определение <<числа вращения>> в других поверхностях довольно 
трудоёмко (так, уже оно само по себе занимает 4 страницы), но даёт лишь <<различающую>> двух гомотопных погружений, 
а не каноничный инвариант отдельного погружения (неканоничный инвариант, конечно, получится, если зафиксировать 
в каждом классе гомотопии какое-нибудь погружение наобум); при этом случай проективной плоскости вовсе выпадает 
из рассмотрения.

Но куда более спорным представляется сам подход, основанный на произвольном выборе векторного поля с нулями --- 
в отсутствие каких-либо указаний на то, что этот случайный выбор не является попросту излишним.
Если некое вычисление проведено в координатах, но ответ не зависит от выбора координат, естественно задаться вопросом 
о том, нельзя ли и всё вычисление провести на инвариантном языке; и нельзя ли возникающий из этого вычисления 
результат сформулировать в инвариантных терминах.

Цель данной заметки --- дать положительный ответ на эти вопросы (применительно к погружениям в поверхность).
Основные результаты работы сформулированы в \S\ref{main results} в случае погружений окружности и обобщены на графы
в \S\ref{graphs}.
Кроме того, в \S\ref{smoothing} доказана теорема сглаживания, устанавливающая связь между топологическими и
гладкими погружениями и формально независимая от остальной части работы (не считая определений в начале \S\ref{circle}).

Довольно странно, что все эти несложные результаты не были получены уже лет 70 назад (точнее, ничего похожего
в литературе мне не удалось найти ни самому, ни путём опроса нескольких экспертов по теме из разных стран). 
Однако в свете сложного результата, опубликованного недавно в данном журнале \cite{P}, кажется, что они всё ещё
(или, наоборот: теперь уже) представляют некоторый интерес.
В случае погружений в $\R P^2$ классификация, полученная в данной работе, 
появилась также в недавнем препринте М.~Ивашковского \cite{I}, где она установлена другим методом 
(см.\ замечания \ref{ivashk} и \ref{relations} ниже).

\section{Погружения окружности} \label{circle}

Под {\it погружением} полиэдра в многообразие мы понимаем непрерывное отображение, которое при ограничении 
на достаточно малую окрестность всякой точки является вложением (т.е. гомеоморфизмом на свой образ).
В частности, погружением является всякое {\it гладкое погружение} между гладкими многообразиями --- дифференцируемое
отображение, имеющее инъективный дифференциал в каждой точке.
{\it (Гладкая) регулярная гомотопия} --- это такая гомотопия $f\:Y\x I\to X$, что отображение $Y\x I\to X\x I$,
$(y,t)\mapsto\big(f(y),t\big)$ --- (гладкое) погружение.

Под {\it поверхностью} мы понимаем $2$-многообразие без края (возможно, неориентируемое и некомпактное).
Несложно показать (см. \S\ref{smoothing}), что всякое погружение из $S^1$ в поверхность $M$ регулярно 
гомотопно гладкому погружению, а регулярно гомотопные гладкие погружения $S^1\to M$ соединяются гладкой 
регулярной гомотопией.

{\it Число вращения} $T(f)$ гладкого погружения $f\:S^1\to\R^2$ --- это число полных оборотов касательного вектора. 
Оно очевидно инвариантно при гладкой регулярной гомотопии, поэтому в силу сказанного выше определено и для 
произвольного погружения $S^1\to\R^2$ и является инвариантом регулярной гомотопии.
{\it Число вращения} $T(f)$ погружения $f\:S^1\to S^2$ --- это класс вычетов по модулю 2, определяемый с помощью 
стереографической проекции: легко видеть, что при проходе погруженной окружности через бесконечность её число 
вращения меняется на двойку.

{\it Самопересечённость} $s(f)$ самотрансверсального гладкого погружения $f$ из $S^1$ в поверхность $M$ --- 
это класс вычетов по модулю 2, противоположный чётности числа двойных точек $f$.
Она очевидно инвариантна при самотрансверсальной регулярной гомотопии, поэтому в силу мульти-0-струйной теоремы 
трансверсальности определена и для произвольного погружения $S^1\to M$ и является инвариантом регулярной 
гомотопии.
(В действительности, $s(f)$ инвариантна и при регулярном бордизме, но нам это не потребуется.)
Как заметил Уитни (см.\ \cite{BP}, \cite{BP2}), для погружений в плоскость $s(f)$ равна чётности $T(f)$; 
следовательно, для погружений в сферу $s(f)=T(f)$.

Если задано нульгомотопное погружение $f$ из окружности в ориентированную поверхность $M$, назовём его 
{\it (ориентированным) числом вращения} $T(f)$ число вращения какого-нибудь поднятия $f$ на универсальное 
накрытие $M$.
(Два разных поднятия $f$ совмещаются сохраняющим ориентацию гомеорфизмом, который изотопен тождественному и,
следовательно, сохраняет число вращения.)
Для нульгомотопного погружения $f$ окружности в произвольную поверхность аналогичное 
{\it (неориентированное) число вращения} $t(f)$ корректно определено по модулю 2, поскольку поднятия $f$ на 
универсальное накрытие совмещаются гомеоморфизмом и, следовательно, имеют одинаковую самопересечённость.
Очевидно, оба этих числа вращения являются инвариантами регулярной гомотопии.

\begin{lemma} \label{self-intersection} 
Для нульгомотопного погружения $f$ из $S^1$ в поверхность $t(f)=s(f)$.
\end{lemma}

\begin{proof} Образ поднятия $\tilde f$ погружения на универсальное накрытие $\tilde M$ можно затащить в сколь угодно 
малый шар $B$ объемлющей изотопией $H_t$ многообразия $\tilde M$ (т.е.\ $H_0=\id$, $H_1\tilde f(S^1)\subset B$).
Если $B$ достаточно мал, ограничение проекции $p\:\tilde M\to M$ на $B$ будет гомеоморфизмом.
Поэтому для погружения $f':=pH_1\tilde f$ неориентированное число вращения $t(f')$ равно его 
самопересечённости $s(f')$.
Но и первое, и вторая остаются неизменными при регулярной гомотопии $pH_t\tilde f$ между $f$ и $f'$.
\end{proof}

\subsection{Ориентированный случай}
Пусть $f$ --- погружение из окружности в ориентированную поверхность $M$, не являющееся нульгомотопным. 
Рассмотрим поднятие $\bar f$ отображения $f$ на накрытие $M_f$ поверхности $M$, соответствующее подгруппе 
$\pi_1(M)$, порождённой $[f]$.
(В качестве базисной точки $M$ берётся образ базисной точки $S^1$.)
Поскольку $\pi_1(M_f)$ конечно порождена, $M_f$ гомеоморфна замкнутой поверхности с конечным числом проколов.
Поскольку $M$ ориентируема, $M_f$ также ориентируема, значит, $M_f$ гомеоморфна сфере с ручками и проколами.
Поскольку $\pi_1(M_f)$ --- циклическая, $M_f$ гомеоморфна кольцу.

Поскольку $M$ ориентирована, на $M_f$ тоже имеется ориентация $o$, которая однозначно задаёт сохраняющее 
ориентацию погружение $\phi_o$ этого кольца в плоскость (ориентированную стандартным образом), при котором 
средняя линия кольца погружается стандартным погружением <<восьмёрка>>, имеющем число вращения $0$. 
В этом случае можно говорить о {\it числе вращения} $T(f)$ погружения $f$, определённом как $T(\phi_o\bar f)$.
Это число вращения очевидно является инвариантом регулярной гомотопии $f$.

Легко видеть, что для $f$, не являющегося нульгомотопным, $T(f)$ меняет знак при смене ориентации $M$ (это 
соответствует переворачиванию погруженного кольца, при котором завитки погружения $\bar f$ меняют знак) и 
остаётся неизменным при смене ориентации $S^1$ (которая реализуется композицией переворачивания кольца и 
зеркальной симметрии плоскости, которая тоже обращает знаки завитков).
Заметим, что эти свойства симметрии подпоритились бы, если бы вместо погружения <<восьмёрка>> мы взяли одно
из стандартных вложений (с числом вращения $1$ или $-1$), что и объясняет наш выбор.

А для нульгомотопного $f$, так же как и для всякого погружения $S^1\to\R^2$, знак $T(f)$ меняется
как при смене ориентации $M$, так и при смене ориентации $S^1$.

\begin{proposition} \label{oriented completeness} Если $T(f)=T(g)$ для гомотопных погружений $f,g$ из окружности в 
ориентированную поверхность, то $f$ и $g$ регулярно гомотопны.

Более того, если заданная гомотопия неподвижна на некоторой дуге окружности, то и регулярную гомотопию можно 
выбрать неподвижной на ней.
\end{proposition}

\begin{proof} 
Первое утверждение сводится ко второму, т.к.\ используя заданную гомотопию, несложно построить регулярную
гомотопию от $f$ к $f'$, совпдающему с $g$ на заданной дуге $J$ и гомотопному $g$ неподвижно на ней.
Второе утверждение достаточно доказать для случая $M=M_f$, поскольку гомотопия между $f$ и $g$ поднимается до 
гомотопии между $\bar f$ и $\bar g$, а регулярная гомотопия между $\bar f$ и $\bar g$ проектируется в
регулярную гомотопию между $f$ и $g$, причём все гомотопии можно считать неподвижными на $J$.
Наконец, утверждение достаточно доказать для гладких погружений, а для них оно согласно h-принципу (см.\ 
\cite{Ad}, \cite{RS}, \cite{RS2}) 
сводится к тому, что если $T(f)=T(g)$, погружения $f$, $g$ имеют фиксированное значение $x$ и фиксированное 
направление $v\in S_xM$ в базисной точке $1\in S^1$ и задают один и тот же элемент в $\pi_1(M,x)$, то их 
сферизованные дифференциалы $sf,sg\:SS^1\to SM$, где $SS^1\cong S^1\sqcup S^1$, задают один и тот же элемент 
в $\pi_1(SM,v)$.
(Здесь $SM$ обозначает сферизацию касательного расслоения, а $S_xM$ --- её слой в точке $x\in M$.)
Случай $M=S^2$ сводится к случаю $M=\R^2$, т.к.\ если для погружений в проколотую сферу $T(f)-T(g)$ чётно, 
легко построить регулярную гомотопию в сфере от $f$ к погружению $f'$ в проколотую сферу с $T(f')=T(g)$.
Значит, можно считать, что $M$ --- плоскость или кольцо, в частности, параллелизуемо.
Тогда $SM\cong M\x S^1$, и $\pi_1(SM,v)\simeq\pi_1(M,x)\x\Z$.
Первая координата $[sf]\in\pi_1(SM,v)$ --- это $[f]\in\pi_1(M,x)$, а вторая --- это $T(f)$.
Поскольку $[f]=[g]$ и $T(f)=T(g)$, утверждение доказано.
\end{proof}

\subsection{Оснащённый случай}
{\it Оснащением} погружения $f$ из $S^1$ в поверхность $M$ называется выбор ориентации ($\Leftrightarrow$ 
тривиализации) в нормальном расслоении $\nu(f)$.
Таким образом, у всякого $f$ с ориентируемым $\nu(f)$ есть ровно два оснащения, а если $\nu(f)$ неориентируемо 
--- то ни одного.
Погружение с выбранным оснащением называется {\it оснащённым погружением}.

Всякое оснащённое погружение $F=(f,\Xi)$ из окружности в поверхность $M$ имеет оснащённое поднятие $(\bar f,\bar\Xi)$ 
на накрытие $M_f$ (если $f$ нульгомотопно, определение $M_f$ тоже имеет смысл и даёт универсальное накрытие).
Поскольку $M_f$ ориентируемо, $\bar\Xi$ задаёт его ориентацию $o$.
Поэтому определено {\it число вращения} $T(F):=T(\phi_o\bar f)$ (в случае, если $M_f$ --- универсальное накрытие,
в качестве $\phi_o$ берётся $\id\:M_f\to M_f$).
Заметим, что если $M$ неориентируемо, а на его ориентирующем накрытии $\hat M$ зафиксирована ориентация, то $F$
имеет однозначное оснащённое поднятие $(\hat f,\hat\Xi)$ на $\hat M$, сохраняющее ориентацию, и $T(F)=T(\hat f)$.

Два гомотопных оснащённых погружения $F,G$ называются {\it оснащённо гомотопными}, если локальные ориентации
поверхности, задающие выбранные оснащения, совмещаются заданной гомотопией.
{\it Оснащённая регулярная гомотопия} --- это просто регулярная гомотопия, являющаяся оснащённой гомотопией.
Очевидно, $T(F)$ является инвариантом оснащённой регулярной гомотопии.

Поскольку оснащённая гомотопия поднимается на орентирующее накрытие, предложение \ref{oriented completeness} имеет

\begin{corollary} \label{framed completeness} Если $T(F)=T(G)$ для оснащённо гомотопных оснащённых 
погружений $F$, $G$ из $S^1$ в поверхность, то $F$ и $G$ оснащённо регулярно гомотопны.

Более того, если заданная гомотопия неподвижна на некоторой дуге окружности, то и регулярную гомотопию можно 
выбрать неподвижной на ней.
\end{corollary}

Из свойств $T(f)$ следует, что $T(f,-\Xi)=-T(f,\Xi)$ и $T(f^*,\Xi^*)=T(f,\Xi)$, где $f^*$ 
обозначает $f$, <<пройденное в обратную сторону>>, т.е.\ предварённое комплексным сопряжением в $S^1\subset\C$, 
а его оснащение $\Xi^*$ имеет ту же ориентацию тотального пространства, что и $\Xi$, но противоположную ориентацию слоёв.

\begin{remark}\label{framing remark}
(a) Если $M=S^2$ или $\R P^2$, то $T(f,\Xi)$ --- класс вычетов по модулю 2, и из соотношения $T(f,-\Xi)=-T(f,\Xi)$
следует, что $T(f,\Xi)$ не зависит от выбора $\Xi$.
Более того, для $M=S^2$ или $\R P^2$ из ориентируемости $\nu(f)$ следует нульгомотопность $f$, откуда 
по определению $T(f,\Xi)=t(f)$; а по лемме \ref{self-intersection} $t(f)=s(f)$.

(б) Если $M\ne S^2,\R P^2$, то из соотношения $T(f,-\Xi)=-T(f,\Xi)$ следует, что абсолютная величина $|T(f,\Xi)|$ 
не зависит от выбора $\Xi$, и стало быть является инвариантом погружения $f$ с ориентируемым $\nu(f)$.
Обозначим её через $T_+(f)$.

(в) Если $M\ne S^2,\R P^2$ и оснащённое погружение $(\phi,\Xi)$ гомотопно $(\phi,-\Xi)$, то для погружений $f$ и $g$, 
гомотопных $\phi$, из $T_+(f)=T_+(g)$ следует, что $f$ и $g$ регулярно гомотопны.
В самом деле, из $T_+(f)=T_+(g)$ следует, что $T(f,\Xi_f)=T(g,\Xi_g)$ для некоторых $\Xi_f$ и $\Xi_g$, 
но в силу нашего предположения произвольные оснащения $f$ и $g$ совмещаются некоторой гомотопией, и утверждение 
вытекает из следствия \ref{framed completeness}.

(г) Допустим теперь, что оснащённое погружение $\Phi:=(\phi,\Xi)$ не гомотопно $(\phi,-\Xi)$.
Тогда для любого погружения $f$, гомотопного $\phi$, можно однозначно выбрать оснащение $\Xi_f$, такое что
$(f,\Xi_f)$ гомотопно $\Phi$.
Следовательно, определён инвариант $T_\Phi(f):=T(f,\Xi_f)$ (хотя он и не определён канонически, т.к. зависит от выбора 
оснащённого погружения $\Phi$).
Согласно следствию \ref{framed completeness} погружения $f$, $g$, гомотопные $\phi$, регулярно гомотопны друг другу,
если и только если $T_\Phi(f)=T_\Phi(g)$.
\end{remark}

\begin{example} \label{framing example}
(а) Если погружение $f\:S^1\to M$ нульгомотопно, то $(f,\Xi)$ оснащённо гомотопно $(f^*,\Xi^*)$.
Если дополнительно $M$ неориентируемо, то протащив $f$ вдоль какого-нибудь листа Мёбиуса, лежащего в $M$, 
мы получим оснащённую регулярную гомотопию от $(f^*,\Xi^*)$ к $(f,-\Xi)$.
В результате получается оснащённая гомотопия $h_t$ от $(f,\Xi)$ к $(f,-\Xi)$.

(б) Пример ненульгомотопного погружения $f$, для которого $(f,\Xi)$ оснащённо гомотопно $(f,-\Xi)$, доставляется вложением
в $M$ края какого-нибудь листа Мёбиуса, лежащего $M$.
Например, такой вид имеет вложение двулистного накрытия над $S^1$ в расслоение бутылки Клейна над $S^1$.

(в) Пример погружения $f$, для которого $(f,\Xi)$ оснащённо гомотопно $(f^*,-\Xi^*)$, но не $(f,-\Xi)$,
даётся слоем расслоения бутылки Клейна над $S^1$.
\end{example}

\begin{proposition} Оснащённое пунктированное погружение $(f,\Xi)$ из $S^1$ в $M$ гомотопно $(f,-\Xi)$ если
и только если $[f]=[g]^n\in\pi_1(M)$ для некоторого пунктированного погружения $g\:S^1\to M$ с неориентируемым 
$\nu(g)$.
\end{proposition}

\begin{proof} Всякая гомотопия $h\:S^1\x I\to M$ погружения $f$ с самим собой поднимается на ориентирующее 
накрытие $\hat M$ (по теореме о накрывающей гомотопии), причём $h$ является оснащённой гомотопией между 
$(f,\Xi)$ и $(f,-\Xi)$, если и только если отображение тора $h'\:S^1\x (I/\partial I)\to M$, через которое 
пропускается $h$, не поднимается на $\hat M$.
Неподнимаемость $h'$ на $\hat M$ в свою очередь равносильна тому, что образ $\pi_1(S^1\x S^1)$ не лежит в
образе $\pi_1(\hat M)$.
Поскольку $[f]\in\pi_1(M)$ лежит в образе $\pi_1(\hat M)$, последнее условие равносильно тому, что
$[f]$ коммутирует с некоторым элементом $\alpha\in\pi_1(M)$, не лежащим в образе $\pi_1(\hat M)$.

Если $[f]=[g]^n$, где $\nu(g)$ неориентируемо, то в качестве такого $\alpha$ можно взять $[g]$.
Обратно, пусть элемент $\alpha$ задан.
Рассмотрим накрытие $\tilde M$ многообразия $M$, соответствующее подгруппе 
$\left<[f],\alpha\right>\subset\pi_1(M)$.
Поскольку $\pi_1(\tilde M)$ конечно порождена, $\tilde M$ гомеоморфно замкнутой поверхности с конечным
числом проколов.
Поскольку $\alpha$ лежит в образе $\pi_1(\tilde M)$, но не лежит в образе $\pi_1(\hat M)$, накрытие $\tilde M\to M$
не пропускается через $\hat M$, так что $\tilde M$ неориентируемо.
Поскольку $[f]$ коммутирует с $\alpha$ в $\pi_1(M)$, группа $\pi_1(\tilde M)$ --- абелева. 
Следовательно, $\tilde M$ --- проективная плоскость или лист Мёбиуса.
Пусть $\tilde g\:S^1\to\tilde M$ --- вложение, представляющее образующую $\pi_1(\tilde M)$, и пусть
$g\:S^1\to M$ --- его композиция с накрытием $\tilde M\to M$.
Поскольку $\nu(\tilde g)$ неориентируемо, то и $\nu(g)$ неориентируемо.
Поскольку $\pi_1(\tilde M)$ --- циклическая, поднятие $\tilde f\:S^1\to\tilde M$ погружения $f$ пунктированно 
гомотопно композиции некоторого накрытия $\phi\:S^1\to S^1$ и вложения $\tilde g$.
Значит, и $f$ пунктированно гомотопно $\phi g$.
\end{proof}

\subsection{Неоснащаемый случай}
Если погружения $f,g\:S^1\to M$ имеют общую начальную точку $p=f(1)=g(1)$, и возникающее отображение букета
$F\:S^1\vee S^1\to M$ является погружением, назовём композицию погружения $\id_{S^1}*\id_{S^1}\:S^1\to S^1\vee S^1$,
задаваемого конкатенацией петель, и погружения $F$ {\it конкатенацией} погружений $f$ и $g$, и обозначим её $f\# g$.

Существует три комбинаторных типа конкатенации, представленные следующими тремя возможностями (но сами не
требующие никакой трансверсальности):
\begin{enumerate}[label=(\Roman*)]
\item $f$ и $g$ трансверсальны в $p$;
\item $f\# g$ самотрансверсально в $p$;
\item $f\# g^*$ самотрансверсально в $p$.
\end{enumerate}

\begin{lemma} \label{concatenation} Если $f$ и $g$ гомотопны, то $s(f\# g)=s(f)+s(g)+w_1\nu(f)$ для типов I, II
и $s(f\# g)\ne s(f)+s(g)+w_1\nu(f)$ для типа III. 
\end{lemma}

Здесь $w_1$ --- класс Штифеля-Уитни, детектирующий ориентируемость расслоения.

\begin{proof} Можно считать, что $f$ и $g$ трансверсальны везде, кроме быть может точки $p=f(1)=g(1)$, 
и самотрансверсальны во всех точках.
Пусть $\bar f$, $\bar g$ --- трансверсальные аппроксимации $f$ и $g$, а $\overline{f\# g}$ --- самотрансверсальная 
аппроксимация $f\# g$; можно считать, что все изменения происходят только в малой окрестности точки $p$. 
Двойные точки $f\# g$ --- это двойные точки $f$, двойные точки $g$ и пересечения между $f$ и $g$.
В случае II точка $p$ даёт одну двойную точку $\overline{f\# g}$, а в случаях I, III --- ноль (или две).
В случае I $p$ даёт одно пересечение между $\bar f$ и $\bar g$, а в случаях II, III --- ноль (или два).
Поэтому в случае III чётность числа двойных точек $\overline{f\# g}$ равна сумме чётностей чисел двойных точек $f$,
двойных точек $g$ и пересечений между $\bar f$ и $\bar g$, а в случаях I, II --- противоположна ей.
Но чётность числа пересечений между $\bar f$ и $\bar g$ --- это произведение классов когомологий по модулю 2,
двойственных к одномерным классам гомологий, представленным $\bar f$ и $\bar g$. 
Но это один и тот же класс (поскольку $\bar f$ и $\bar g$ гомотопны), и его квадрат в 
$H^2_c(M)\simeq H^2_c\big(\nu(f)\big)$ переводится изоморфизмом Тома в $w_1\nu(f)$.
(На геометрическом языке: $w_1\nu(f)$ есть класс Эйлера по модулю 2, так что он равен чётности числа нулей общего сечения 
$\nu(f)$.)
\end{proof}

Если $f\:S^1\to M$ --- гладкое погружение и $o\:S^1\to M$ --- вложение на маленькую гладкую окружность, имеющую 
сонаправленное касание с $f$ в точке $f(1)=o(1)$, будем говорить, что конкатенация $f\# o$ получена из $f$ 
{\it добавлением завитка}; собственно {\it завитком} погружения $f\# o$ будем называть дугу $o(S^1\but\{1\})$
его образа.
Если $M$ ориентировано или $f$ оснащено, у завитка имеется знак, определяемый тем, с какой стороны от образа $f$
он расположен.

\begin{lemma} \label{Moebius}
Пусть $f$ --- вложение из $S^1$ в среднюю линюю (= нулевое сечение) листа Мёбуиса $M$ и пусть $\phi\:S^1\to M$ 
--- погружение, полученное из $f$ 1) добавлением одного завитка, 2) репарамертризацией, в результате которой  
$\phi(1)$ попадает в завиток, и 3) регулярной гомотопией, возникающей под действием изотопии листа Мёбиуса 
с носителем в некоторой окрестности $U$ замыкания завитка, в результате которой $\phi(1)$ перемещается в $f(1)$
таким образом, что $\phi$ и $f$ имеют в этой точке противонаправленное касание, причём других пересечений между 
$\phi$ и $f$ внутри $U$ нет.
Тогда $[(sf^*)^*]=[s\phi]\in\pi_1\big(SM,s\phi(1)\big)$.  
\end{lemma}

Доказательство несложно; заметим лишь, что при замене листа Мёбиуса на кольцо утверждение становится неверным.

\begin{proposition} \label{unoriented completeness} Если $s(f)=s(g)$ для гомотопных погружений $f,g$ из $S^1$ в 
поверхность, причём $\nu(f)$ неориентируемо, то $f$ и $g$ регулярно гомотопны.

Более того, если заданная гомотопия неподвижна на некоторой точке окружности, то и регулярную гомотопию можно 
выбрать неподвижной на ней.
\end{proposition}

\begin{proof} Как и в доказательстве предложения \ref{oriented completeness}, достаточно разобрать случай 
с неподвижной точкой $p$.
Погружения $f$ и $g$ можно подправить регулярной гомотопией так, чтобы они стали гладкими и чтобы точка $p$
стала бы для них точкой сонаправленного касания, квадратичного в том смысле, что оно локально выглядит 
как касание графиков $y=x^2$ и $y=-x^2$ (а не $y=x^3$ и $y=-x^3$). 
Достаточно было бы показать, что сферизованный дифференциал $sf$ пунктированно гомотопен $sg$, или,
другими словами, что петля $sf\cdot(sg)^*$ нульгомотопна.
Пусть $h$ --- гладкая аппроксимация погружения $f\# g^*$, регулярно гомотопная ему и дважды касающаяся $f$ 
в точке $p$.
Поскольку $\nu(f)$ неориентируемо, из леммы \ref{Moebius} следует, что $sh$ пунктированно гомотопен 
$sf\cdot(sg)^*$.
Поэтому было бы достаточно показать, что $sh$ нульгомотопен.

Поскольку $\nu(f)$ неориентируемо и $s(f)=s(g)=s(g^*)$, по лемме \ref{concatenation} (для конкатенации типа III)
$s(h)=0$.
Тогда по лемме \ref{self-intersection} $T(h,\Xi)$ чётно для любого оснащения $\Xi$.
Добавим к $h$ чётное число завитков одного знака (знаки определяются выбором $\Xi$), так чтобы 
полученное оснащённое погружение имело $T(h',\Xi')=0$.
Тогда по следствию \ref{framed completeness} $h'$ регулярно гомотопно восьмёрке во вложенном диске,
и в частности сферизованный дифференциал $sh'$ нульгомотопен.
Можно считать, что $h'$ --- сглаживание $f'\# g^*$, где $f'$ получено из $f$ добавлением чётного числа завитков, 
не задевающих точку $p$.
Тогда $sf'$ пунктированно гомотопен $sg$, и, значит, по h-принципу (см.\ 
\cite{Ad}, \cite{RS}, \cite{RS2}) $f'$ регулярно гомотопно $g$, оставляя
неподвижной точку $p$ (и даже касательное направление в ней).
С другой стороны, $f'$ регулярно гомотопно $f$, поскольку половину завитков можно протащить регулярной 
гомотопией вдоль всего образа $f$, в результате чего они сменят свой (локальный) знак и сократятся с оставшимися.
Эту регулярную гомотопию можно подправить так, чтобы она оставляла неподвижной точку $p$ (но, конечно, не касательное
направление в ней). 
\end{proof}

\begin{example}
Заменить в формулировке предложения \ref{unoriented completeness} <<точку>> на <<дугу>> не получится. 
Действительно, рассмотрим вложение $f$ из $S^1$ в среднюю линию листа Мёбиуса, и пусть $g$ --- погружение,
полученное из $f$ добавлением пары завитков, не пересекающих какую-нибудь дугу $J$ на образе $f$.
Тогда $s(f)=s(g)$, но если бы $f$ и $g$ были регулярно гомотопны, оставляя неподвижной $J$,
то конкатенация $f\# g^*$ была бы регулярно гомотопна $f\# f^*$.
Но $T(f\# f^*,\Xi)=\pm 1$ в зависимости от выбора $\Xi$, тогда как при подходящем знаке завитков
$T(f\# g^*,\Xi')=\pm 3$ в зависимости от выбора $\Xi'$, так что такой регулярной гомотопии не существует. 
\end{example}

\subsection{Основные результаты} \label{main results}

Следующие результаты --- прямые следствия предложений, доказанных выше. 

\begin{theorem} Погружения $f$, $g$ из $S^1$ в ориентированную поверхность регулярно гомотопны, если и только если
$f$ и $g$ гомотопны и $T(f)=T(g)$.
\end{theorem}

\begin{theorem} \label{unoriented classification} 
Погружения $f$, $g$ из $S^1$ в поверхность $M$ регулярно гомотопны, если и только если
$f$ и $g$ гомотопны, причём 
\begin{itemize}
\item если $\nu(f)$ ориентируемо, $T(f,\Xi_f)=T(g,\Xi_g)$ для каких-нибудь оснащений $\Xi_f$, $\Xi_g$,
согласованных относительно какой-нибудь гомотопии между $f$, $g$;
\item если $\nu(f)$ неориентируемо, $s(f)=s(g)$.
\end{itemize}
Более того, условие согласованности однозначно задаёт $\Xi_g$ по $\Xi_f$, если и только если $[f]\in\pi_1(M)$ 
не является степенью $[\phi]$ с неориентируемым $\nu(\phi)$; при неоднозначности оснащения $\Xi_g$ его изменение
выражается в смене знака $T(g,\Xi_g)$.
\end{theorem}

В случае $M=\R P^2$ формулировку теоремы \ref{unoriented classification} можно упростить, поскольку
всякое погружение $f\:S^1\to\R P^2$ либо имеет неориентируемое $\nu(f)$, либо нульгомотопно, причём в последнем
случае $T(f,\Xi)=s(f)$ (см.\ замечание \ref{framing remark}(а)):

\begin{corollary} \label{RP2} Погружения $f,g\:S^1\to\R P^2$ регулярно гомотопны, если и только если
$f$ и $g$ гомотопны и $s(f)=s(g)$.
\end{corollary}

\begin{remark} \label{ivashk}
В недавней работе М. Ивашковского \cite{I} получено другое доказательство следствия \ref{RP2}, основанное 
на сведении общего случая к конкретным примерам с использованием движений Райдемайстера.
Автор узнал о результате Ивашковского из комментариев его научного руководителя Е. Кудрявцевой к приведённому 
выше доказательству теоремы \ref{unoriented classification}.
\end{remark}

\begin{remark} \label{universal cover}
Если погружения $f$ и $g$ пунктированно гомотопны и имеют в точке $f(1)=g(1)$ противоположные направления
(чего легко достичь их пунктированной регулярной гомотопией), то погружение $f\# g^*$ нульгомотопно, 
и условия $T(f)=T(g)$, $T(f,\Xi_f)=T(g,\Xi_g)$ и $s(f)=s(g)$ можно эквивалентно переформулировать в виде 
$T(f\# g^*)=0$, $T(f\# g^*,\Xi_f\#\Xi_g^*)=0$ и $t(f\# g^*)=w_1\nu(f)$, где $T$ и $t$ определены уже в терминах 
обычного числа вращения на универсальном накрытии.
(Действительно, первое утверждение следует из того, что $T(f\#g^*)$ можно вычислить не только на универсальном
накрытии, но и на накрытии $M_f$, где оно очевидно складывается из $T(f)$ и $-T(g)$, второе проверяется
аналогично, а третье вытекает из леммы \ref{self-intersection} и леммы \ref{concatenation} для типов I, II.)
\end{remark}

\begin{remark} Для произвольных погружений $f$ и $g$ (вообще говоря, имеющих разные значения 
в точке $1$) вместо конкатенации $f\#g^*$ определена {\it связная сумма} $f\#_p g^*$ вдоль погруженного 
пути $p\:[a,b]\to M$, такого что $p(a)=f(1)$ и $p(b)=g(1)$, причём $f$, $p$ и $g$ вместе задают погружение 
$S^1\cup_{1=a} I\cup_{b=1} S^1$ в $M$.
Если $f$ и $g$ соединяются гомотопией $h_t$, и путь $p$ гомотопен пути $h_t(1)$ по модулю концов, то 
связная сумма $f\#_p g^*$ нульгомотопна.
Если $f$ и $g$ снабжены оснащениями $\Xi_f$ и $\Xi_g$ (что выполнено автоматически, если на $M$ задана ориентация),
то можно различить 4 типа связных сумм: $(++)$, $(+-)$, $(-+)$ и $(--)$, в зависимости от согласованности ориентации 
пути $p$ с оснащениями $\Xi_f$ и $\Xi_g$ (понимаемыми здесь как ориентации слоёв $\nu(f)$ и $\nu(g)$).
Деление на {\it знакопеременые} ($+-$ и $-+$) и {\it знакопостоянные} типы ($++$ и $--$) имеет смысл и без выбора 
оснащений (и даже когда их не существует), если между $f$ и $g$ задана гомотопия $h_t$, поскольку в этом случае 
можно предполагать согласованность оснащений (если они существуют) относительно $h_t$ или по крайней мере 
согласованность локальных ориентаций в точках $f(1)$ и $g(1)$ (которые можно понимать как локальные оснащения 
$f$ и $g$ в окрестностях этих точек) относительно пути $h_t(1)$. 
В этих предположениях легко видеть, что конкатенации $f\# g^*$ типов I и II (как в замечании \ref{universal cover})
реализуются знакопеременными связными суммами, а типа III (как в доказательстве предложения
\ref{unoriented completeness}) --- знакопостоянными.
Прямые вычисления показывают, что не только в этих случаях, но и для произвольных знакопеременных связных сумм 
$T(f\#_p g^*)=T(f)-T(g)$, $T(f\#_p g^*,\Xi_f\#_p\Xi_g^*)=T(f,\Xi_f)-T(g,\Xi_g)$ и 
$t(f\#_p g)=s(f)-s(g)+w_1\nu(f)$; а для любых знакопостоянных связных сумм $T(f\#_p g^*)=T(f)-T(g)\pm 1$, 
$T(f\#_p g^*,\Xi_f\#_p\Xi_g^*)=T(f,\Xi_f)-T(g,\Xi_g)\pm 1$ и $t(f\#_p g)\ne s(f)-s(g)+w_1\nu(f)$.
Здесь знак $\pm 1$ зависит уже от конкретного типа связной суммы ($++$ или $--$), причём конкатенациям типа
III соответствует только один из двух знакопостоянных типов (какой именно --- зависит от выбора ориентации $M$ или 
согласованных оснащений).
\end{remark}

\section{Погружения графов} \label{graphs}

Под {\it графом} будет пониматься конечный одномерный CW-комплекс (иными словами, допускаются петли и кратные рёбра).

Зафиксируем какой-нибудь остов $T$ связного графа $\Gamma$ (т.е.\ подграф $\Gamma$, являющийся деревом и содержащий все 
вершины $\Gamma$). 
Под {\it опорным циклом} $\Gamma$ (относительно $T$) понимается единственный цикл любого подграфа графа $\Gamma$, 
полученного из $T$ добавлением ровно одного ребра.

\begin{theorem} \label{oriented graph}
Пусть $f,g$ --- погружения связного графа $\Gamma$ в ориентированную поверхность, гомотопные друг другу. 
Тогда $f$ и $g$ регулярно гомотопны, если и только если они задают одинаковые циклические порядки рёбер 
в вершинах $\Gamma$, и для каждого опорного цикла $C$ выполнено $T(f|_C)=T(g|_C)$.
\end{theorem}

Теорема \ref{oriented graph} --- непосредственное следствие теоремы \ref{unoriented graph} ниже.

Всякий элемент $\alpha\in H_1(\Gamma;\,\Z/2)$ представляется эйлеровым подграфом в $\Gamma$ (т.е.\ подграфом, 
в котором степень каждой вершины чётна), который в свою очередь представляется погружением в $\Gamma$ 
дизъюнктного объединения окружностей.
Рассмотрев связную сумму этих погруженных в $\Gamma$ окружностей вдоль каких-нибудь путей (последовательно
добавляя по одной окружности), мы получим одно погружение $f\:S^1\to\Gamma$, представляющее $\alpha$.

\begin{theorem} \label{unoriented graph}
Погружения $f,g$ связного графа $\Gamma$ в поверхность $M$ регулярно гомотопны, если и только если 
они соединяются гомотопией $h_t$, такой что:
\begin{enumerate}
\item $f$ и $g$ задают одинаковый циклический порядок рёбер в каждой вершине $\Gamma$ относительно некоторых локальных 
ориентаций $M$ в образах этой вершины, согласованных относительно $h_t$;
\item для всякого опорного цикла $C$, такого что $\nu(f|_C)$ ориентируемо, 
$T(f|_C,\Xi_f)=T(g|_C,\Xi_g)$ для некоторых оснащений $\Xi_f$, $\Xi_g$, совмещаемых $h_t|_C$;
\item для одного опорного цикла $Z$, такого что $\nu(f|_Z)$ неориентируемо, $s(f|_Z)=s(g|_Z)$ (если такой цикл
существует), а для остальных опорных циклов $C$, таких что $\nu(f|_C)$ неориентируемо,
$T(f\phi_C,\Xi_f)=T(g\phi_C,\Xi_g)$, где $\phi_C$ --- какое-нибудь погружение из $S^1$ в $\Gamma$, представляющее 
$[C]+[Z]\in H_1(\Gamma;\,\Z/2)$, а $\Xi_f$, $\Xi_g$ --- какие-нибудь оснащения, совмещаемые $h_t\phi_C$.
\end{enumerate}
\end{theorem}

\begin{proof}
Если цикл $Z$ существует, пусть $v$ --- какая-нибудь его вершина.
Используя заданную гомотопию, легко постороить регулярную гомотопию от $f$ к $f'$, совпадающему с $g$ на $v$
и гомотопному $g$, оставляя неподвижной $v$.
Заметим, что локальные ориентации, перенесённые из $f(v)$ в $g(v)=f'(v)$ исходной гомотопией и построенной
регулярной гомотопией совпадают. 
Теперь совместим регулярной гомотопией, неподвижной на $v$, погружения $f'|_Z$ и $g|_Z$, используя предложение
\ref{unoriented completeness}.
Построенная регулярная гомотопия продолжается до регулярной гомотопии всего графа $\Gamma$ между $f'$ и 
погружением $f''$, совпадающим с $g$ на $Z$. 
При этом $f''$ гомотопно $g$, оставляя неподвижной $v$.

Далее совместим регулярной гомотопией ограничения $f''$ и $g$ на малую окрестность $Z$ в $\Gamma$, используя
условие (1), а после этого совместим также и их ограничения на малую окрестность $T$ в $\Gamma$, используя 
условие (1), рассмотрев какое-нибудь симплициальное сдавливание дерева $T$ на его поддерево $T\cap Z$ 
(т.е.\ последовательность графов $\Gamma_1,\dots,\Gamma_n$, в которой $\Gamma_1=T$, $\Gamma_n=T\cap Z$ 
и каждый $\Gamma_{i+1}$ получен из $\Gamma_i$ удалением какой-нибудь вершины степени 1 вместе с содержащим 
её рёбром) и ведя индукцию по этому сдавливанию (т.е.\ по номеру $i$ графа $\Gamma_i$).
(Проблем с совпадением локальных ориентаций в вершинах не возникает, поскольку все они переносятся по дереву
$T$ из вершины $v$.
Каждый шаг индуктивного процесса совмещения основан на том, что два погружения отрезка $[0,1]$, совпадающие на 
$[0,\epsilon]$, регулярно гомотопны, оставляя неподвижным $[0,\epsilon]$.
При этом лишние завитки выталкиваются из отрезка через точку $1$.)  
Это даёт погружение $f'''$, регулярное гомотопное $f$ и совпадающее с $g$ на некоторой окрестности $T\cup Z$.
Если же цикла $Z$ не существует, то аналогично строится $f'''$, регулярное гомотопное $f$ и совпадающее с $g$ на 
некоторой окрестности $T$.
В обоих случаях ясно, что $f'''$ гомотопно $g$, оставляя неподвижной некоторую окрестность $T$.

Рассмотрим теперь какой-нибудь опорный цикл $C$, отличный от $Z$. 
Он задаётся ребром $E$, не входящим в $T$, причём $f'''$ уже совпадает с $g$ на $C\but J$, где $J$ --- некоторая 
замкнутая дуга во внутренности ребра $E$.
Если $\nu(f|_C)$ ориентируемо, по следствию \ref{framed completeness} cуществует регулярная гомотопия между
$f'''|_C$ и $g|_C$, неподвижная вне $J$.
(Проблем с оснащениями не возникает, поскольку они задаются локальными ориентациями в вершинах графа, 
которые уже совпали.)
Если $\nu(f|_C)$ неориентируемо, по тому же следствию \ref{framed completeness} cуществует регулярная гомотопия между 
$f'''\phi_C$ и $g\phi_C$, неподвижная вне $\phi^{-1}(J)$.
Но это то же самое, что и регулярная гомотопия ребра $E$ c носителем в $J$.
Поскольку построенные регулярные гомотопии имеют непересекающиеся замкнутые носители, их одновременное применение
даёт требуемую регулярную гомотопию между $f'''$ и $g$.
\end{proof}

\begin{remark}
Из доказательств предложений \ref{oriented completeness} и \ref{unoriented completeness} и теоремы \ref{unoriented graph} 
ясно, что построенные в них регулярные гомотопии гомотопны (оставляя неподвижными верхнее и нижнее основания цилиндра)
заданной гомотопии между $f$ и $g$.
\end{remark}

\subsection{Классификация} \label{graph classification}

Покажем, как из теоремы \ref{unoriented graph} извлечь полный (но неканоничный) инвариант погружений 
$\Gamma\to M$ (без какой-либо дополнительной структуры) и докажем, что все его значения реализуются.

Легко видеть, что произвольный набор циклических порядков рёбер в вершинах реализуется погружением графа
во вложенный диск $D$ в $M$.
Все дальнейшие модификации будут производиться только на дополнении к некоторой окрестности $T$ в $\Gamma$, 
так что циклические порядки рёбер в вершинах уже не изменятся, если не менять ориентацию диска $D$.

Поскольку $\Gamma$ гомотопически эквивалентен $\Gamma/T$, заданный свободный класс гомотопии $\alpha\in[\Gamma,M]$
задаёт (неоднозначно) пунктированное отображение $\Gamma/T\to M$, где базисная точка $M$ берётся в $D$.
Подправим построенное выше погружение $\Gamma\to D$ на каждом ребре $\Gamma\but T$ так, чтобы реализовать
класс соответствующего лепестка букета $\Gamma/T$ в $\pi_1(M)$.
Полученное в результате погружение $f$ реализует $\alpha$.

Если цикл $Z$ существует, то $s(f|_Z)$ может принимать два значения, и оба они реализуются, т.к.\ к
$f|_Z$ можно добавлять завитки (на ребре $Z$, не входящем в $T$).

Мы можем зафиксировать ориентацию $o$ диска $D$ произвольным образом, что автоматически задаст оснащения $\Xi^o$ 
всех погруженных циклов $f|_C$, для которых $\nu(f|_C)$ ориентируемо, --- и вообще всех композиций $\phi f$, где 
$\phi\:S^1\to\Gamma$ --- погружение, такое что $\nu(\phi f)$ ориентируемо.
Если пара $F:=(f,o)$ не гомотопна паре $(f,-o)$ (что, конечно, является свойством гомотопического класса $\alpha$),
то для всякого погружения $g$, гомотопного $f$, существует единственная ориентация $o_g$ диска $D$, такая что
пара $(f,o)$ гомотопна паре $(g,o_g)$.
В этом случае корректно (хотя и не канонично) определены целочисленные инварианты $T_F(g|_C):=T(g|_C,\Xi^{o_g})$ и 
$T_F(\phi_C g):=T(\phi_C f,\Xi^{o_g})$, где $C$ --- опорный цикл $\Gamma$, такой что $\nu(g|_C)$ ориентируемо,
а $\phi_C\:S^1\to\Gamma$ --- погружение как в пункте (3) формулировки теоремы \ref{unoriented graph}.
Все возможные наборы значений этих инвариантов реализуются, поскольку к погруженным рёбрам $\Gamma\but (T\cup Z)$
можно добавлять завитки (независимо для разных рёбер, число которых равно числу рассматриваемых целочисленных
инвариантов).
По теореме \ref{unoriented graph} комбинация этих целочисленных инвариантов с $s(f|_Z)$ (если цикл $Z$ имеется) 
и циклическими порядками рёбер в вершинах будет в этом случае полным инвариантом регулярной гомотопии погружений
$G\to M$, причём, как показано выше, этот инвариант сюръективен. 

Если же пара $F:=(f,o)$ гомотопна паре $(f,-o)$, то инварианты $T_F(g|_C)$ и $T_F(\phi_C g)$ корректно определены
лишь с точностью до одновременной смены знака, при которой также переворачиваются все циклические порядки рёбер
в вершинах. 
(На $s(f|_Z)$ эта смена знака никак не влияет.)
Таким образом, набор указанных инвариантов, профакторизованный по одновременной смене знака и дополненный классом 
$\bmod2$ вычетов $s(f|_Z)$, является в данном случае полным инвариантом регулярной гомотопии погружений $G\to M$ 
(по теореме \ref{unoriented graph}), причём все его значения реализуются по тем же соображениям, что и в предыдущем случае.

\begin{example} Сравним инварианты нульгомотопных погружений $S^1\vee S^1$ и $S^1\sqcup S^1$ в лист Мёбиуса $M$.
Согласно замечанию \ref{framing remark}(в) и примеру \ref{framing example}(а) нульгомотопные погружения $f\:S^1\to M$ 
классифицируются $T_+(f)$ --- абсолютной величиной числа вращения.
Поэтому и для погружений $S^1\sqcup S^1\to M$ можно говорить о числе завитков на каждой окружности, но не имеет смысла
говорить об их знаках.
Однако для погружений $S^1\vee S^1\to M$ о знаках завитков можно кое-что сказать. 
Хотя они и не определены как глобальные инварианты, можно говорить о том, одинаковые они или разные на двух окружностях
букета (в силу сказанного выше).

Таким образом, некоторые инварианты погруженных графов не сводятся к инвариантам погруженных окружностей --- хотя
с учётом оснащений такая редукция имеет место (согласно теореме \ref{unoriented graph}).
\end{example}

\begin{remark} \label{relations}
Исходной мотивировкой данной работы было упрощение аналога теоремы \ref{unoriented graph}, полученного 
в \cite{P}.
Замечание \ref{framing remark} и \S\ref{graph classification} распространяют это упрощение и на главу 3 
диссертации \cite{P2}, не вошедшую в статью \cite{P}.  
В работе \cite{I} приводится другое доказательство теоремы \ref{unoriented graph} в случае, когда $M=\R P^2$, 
а $\Gamma$ --- букет окружностей.
\end{remark}

\section{Сглаживание погружений}\label{smoothing}

\begin{lemma} \label{Antoine}
Всякое вложение $f\:I\to\R^2$ переводится в гладкое вложение изотопией $\R^2$ с носителем в 
заданной окрестности $f(I)$.
\end{lemma}

\begin{proof} Если заменить <<гладкое>> на <<кусочно-линейное>>, то это --- известная теорема Шёнфлиса--Антуана
\cite{An}*{p.~21, Assertion 17} 
(см.\ также \cite{Mo}*{Theorem 10.8} и \cite{Ke}*{Theorem 4.3}; наиболее простое доказательство 
получается по аналогии с \cite{Bi}*{proof of Theorem III.6.A}).
Но кусочно-линейную дугу несложно сгладить изотопией $\R^2$ с носителем в заданной окрестности дуги.
\end{proof}

\begin{lemma} \label{Antoine2}
Всякое вложение $f\:[-2,2]\to\R^2$, такое что $f|_{[-2,-1]\cup[1,2]}$ --- гладкое вложение, 
изотопно гладкому вложению 
изотопией с носителем в $[-1,1]$.
\end{lemma}

Под {\it изотопией} здесь понимается гомотопия в классе вложений.
Доказательство не даёт изотопии $\R^2$, накрывающей изотопию отрезка, но зато оно просто.

\begin{proof} Пусть $g\:[-2,2]\to\R^2$ --- стандартное включение на отрезок $[-2,2]$ оси абсцисс, и пусть
$H\:\R^2\to\R^2$ --- диффеоморфизм, переводящий $f|_{[-2,-1]\cup[1,2]}$ в $g|_{[-2,-1]\cup [1,2]}$ и
такой, что $Hf([-1,1])$ не пересекает лучи $[-\infty,-2]$ и $[2,\infty]$ оси абсцисс.
Пусть $h_t\:[-2,2]\to\R^2$ --- коническая изотопия между $Hf$ и $g$, определённая по формуле
$h_t(x)=(1-t)Hf(\frac{x}{1-t})$ при $|x|<1-t$, а иначе $h_t(x)=g(x)$.
Тогда $H^{-1}h_t$ --- искомая изотопия.
\end{proof}

\begin{lemma} \label{Schoenflies}
Всякое вложение $f$ из $S^1$ в поверхность $M$ регулярно гомотопно гладкому погружению.

Более того, если $f$ гладко вкладывает некоторую дугу, то регулярную гомотопию можно считать неподвижной на ней.
\end{lemma}

\begin{proof} Первое утверждение сводится ко второму, если рассмотреть дисковую окрестность $D$ какой-нибудь
точки $f(S^1)$ в $M$ и достаточно малую дугу $J\subset S^1$, такую что $f(J)\subset D$, и применить к 
ней лемму \ref{Antoine}.

Докажем второе утверждение.
Пусть $J\subset S^1$ --- замкнутая дуга, внутренность которой содержит замыкание дополнения до заданной дуги.
Применим лемму \ref{Antoine2} к поднятию $f|_J$ на универсальное накрытие $M$ (если это плоскость; а если это сфера,
выколем из неё точку, не попадающую в образ $J$).
Полученная изотопия проектируется в искомую регулярную гомотопию в $M$.
\end{proof}

\begin{theorem} \label{smoothing imm}
Всякое погружение $f$ из $S^1$ в поверхность $M$ регулярно гомотопно гладкому погружению.

Причём регулярную гомотопию можно взять неподвижной на заданной точке $S^1$ или, если ограничение $f$ гладко
погружает некоторую дугу, --- то на этой дуге.
\end{theorem}

\begin{proof}
Пусть задано погружение $f\:S^1\to M$.
В силу компактности $S^1$ она покрывается конечным числом открытых дуг $J_1,\dots,J_m$, таких что 
каждое $f|_{J_i}$ --- вложение.
Можно считать, что непусты лишь пересечения вида $K_i:=J_i\cap J_{i+1}$ (сложение по модулю $m$).
Уменьшая при необходимости дуги $K_i$, можно считать, что образы их замыканий не пересекаются, и что
каждая из них содержится во вложенном диске $D_i\subset M$.
Тогда из леммы \ref{Antoine} следует, что существует объемлющая изотопия $H_t$, $H_0=\id_M$, переводящая каждую 
вложенную дугу $f(K_i)$ в гладко вложенную дугу $g(K_i)$, где $g=H_1f$.
Рассмотрим замкнутые гладко вложенные дуги $L_1,\dots,L_m$, не пересекающие друг друга и такие, что каждая $L_i$ 
трансверсально пересекает $g(K_i)$ в одной точке $g(p_i)$ и не имеет других точек пересечения с $g(S^1)$.
Рассмотрим также их дизъюнктные окрестности $U_i$, гомеоморфные открытому диску и такие, что каждая $U_i$
пересекает $g(S^1)$ по дуге, лежащей в $g(K_i)$.
Рассмотрим поверхность $M'=M\but(\partial L_1\cup\dots\cup\partial L_m)$, полученную из $M$ прокалыванием в
$2n$ точках, а также её накрытие $M'_g$, соответствующее подгруппе $\pi_1(M')$, порождённой $[g]$.
Покажем, что поднятие $\bar g$ погружения $g$ в $M'_g$ является вложением.
 
Действительно, пусть $\bar g(x)=\bar g(y)$, где $x\ne y$.
Тогда $x$ и $y$ не лежат одновременно ни в какой дуге $J_i$.
Следовательно, найдутся $i$ и $j$, такие что 0-сфера $\{p_i,p_j\}$ зацеплена с 0-сферой $\{x,y\}$ в $S^1$,
причём $x,y\notin K_i\cup K_j$.
Тогда $z:=g(x)=g(y)$ содержится в $N:=M'\but (U_i\cup U_j)$.
Факторпространство $M'/N$ гомотопически эквивалентно букету двух окружностей, причём композиция $g$ с 
факторотображением $q\:M'\to M'/N$ переводит две дуги $I_1$, $I_2$, на которые $x$ и $y$ разделяют $S^1$ 
(с ориентацией, индуцированной из $S^1$), в петли, представляющие обе образующие группы $\pi_1(M'/N)$.
При этом сама эта композиция $qg$ представляет произведение этих образующих.
Поэтому её класс порождает подгруппу, не содержащую классов петель $qg|_{I_1}$ и $qg|_{I_2}$.
Следовательно, классы петель $g|_{I_1}$ и $g|_{I_2}$ также не лежат в подгруппе $\pi_1\big(M',z)$, 
порождённой классом $g$.
Но тогда они не поднимаются в $M'_g$, что противоречит нашему предположению.

Итак, $\bar g$ является вложением. 
Рассмотрим теперь $M''=M\but H_1^{-1}(\partial L_1\cup\dots\cup\partial L_m)$, а также её накрытие $M''_f$, 
соответствующее подгруппе $\pi_1(M'')$, порождённой $[f]$.
Поднятие $\bar f$ погружения $f$ в $M''_f$ переводится в $\bar g$ гомеоморфизмом $M''_f\to M'_g$ накрывающих
пространств, лежащим над гомеоморфизмом $H_1|_{M''}\:M''\to M'$ баз, и потому тоже является вложением.
Значит, по лемме \ref{Schoenflies} оно соединяется регулярной гомотопией $h_t$ с гладким погружением, причём 
если $f$ гладко вкладывает некоторую дугу, то $h_t$ неподвижна на ней.
Если задана точка $1\in S^1$, её можно сделать неподвижной, заменив $h_t$ на $\phi_t^{-1}h_t$, где $\phi_t$ --- 
изотопия от тождества к какому-нибудь диффеоморфизму $M''_f$, такая что $\phi_t\bar f(1)=h_t(1)$.
Но тогда композиция $h_t$, проекции на $M''$ и включения $M''\to M$ является искомой регулярной гомотопией от $f$ 
к гладкому погружению.
\end{proof}

\begin{theorem} \label{smoothing hom}
Если два гладких погружения $f,g\:S^1\to M$ регулярно гомотопны, то они гладко регулярно гомотопны. 

Более того, гладкая регулярная гомотопия будет неподвижной на заданной точке или дуге, если это выполнено для заданной
регулярной гомотопии. 
\end{theorem}

\begin{proof}
Пусть задана регулярная гомотопия $h_t\:S^1\to M$ между $f$ и $g$.
В силу компактности цилиндра $S^1\x I$ он покрывается конечным числом открытых множеств вида $J_i\x I_i$, $i=1,\dots,m$, 
таких что для каждого $i$, при любом $t\in I_i$ отображение $h_t|_{J_i}$ --- вложение.
Поэтому существует $\eps>0$, такое что $h_t(x)\ne h_t(x+\delta)$ (сложение в группе $S^1$) при $\delta<\eps$ и любом
$t\in I$.
Пусть $e_y\:T_y M\to M$ --- эспоненциальное отображение, заданное некоторой полной римановой метрикой на $M$. 
Оно инъективно на шаре радиуса $r(y)>0$ c центром в начале координат, где функция $r(y)$ непрерывна, и следовательно
ограничена снизу некоторым $r>0$ для всех $y$ из образа $S^1\x I$ (в силу его компактности).
Поэтому для достаточно малых $\delta$ корректно определён ненулевой касательный вектор 
$v_t^\delta(x):=e_{h_t(x)}^{-1}\big(h_t(x+\delta)\big)/\delta$ в точке $x$.
В силу равномерной непрерывности $df_x(1)\:S^1\to TM$ (т.е.\ ограничения $df\:TS^1\to TM$ на единичные вектора) 
функции $v_0^\delta(x)$ равномерно сходятся к $df_x(1)$ при $\delta\to 0$.
Поэтому при достаточно малых $\delta$ линейная гомотопия между $v_0^\delta(x)$ и $df_x(1)$ не задевает начала
координат.
Зафиксировав такое $\delta$, обозначим $v_t^\delta(x)/||v_t^\delta(x)||$ через $\phi_t(x)$.
Тогда $\phi_t\:S^1\to SM$ вместе с гомотопией между $\phi_0$ и $sf$ и аналогичной гомотопией между $\phi_1$ и $sg$
задают гомотопию между сферизованными дифференциалами $sf$ и $sg$, и требуемое утверждение вытекает из h-принципа 
(см.\ \cite{Ad}, \cite{RS}, \cite{RS2}).
Точнее, в случае с неподвижной дугой полученная гладкая регулярная гомотопия будет неподвижна лишь на некоторой точке 
и на касательном векторе в этой точке, но используя это, несложно сделать её неподвижной и на всей дуге; а в случае
с неподвижной точкой мы можем использовать путь $\phi_t(1)$, чтобы построить диффеотопию $G_t$ многообразия $M$,
неподвижную на $h_0(1)$ и такую, что $sG_t\phi_t(1)=\phi_0(1)$.
\end{proof}

\begin{remark} Обобщение теорем \ref{smoothing imm} и \ref{smoothing hom} на гладкие погружения графов 
(в смысле \cite{Mu}) предоставляется читателю.
\end{remark}

\section{Благодарности}
Автор благодарен П. Ахметьеву, Е. Кудрявцевой, Д. Пермякову и О. Фролкиной за полезные обсуждения.
В частности, идея рассмотреть накрытие $M_f$ исходит от Е. Кудрявцевой, а П. Ахметьев и Д. Пермяков указали ошибки
в ранних версиях заметки.

\end{document}